\documentclass[11pt]{amsart}

\usepackage{calrsfs}

\usepackage{ amsthm,amsfonts,amsbsy,amssymb,amsmath,amscd}
\usepackage{amsmath}
\usepackage[all]{xy}
\usepackage{xr-hyper}
\usepackage[colorlinks=true]{hyperref}

\usepackage{verbatim}
\usepackage[dvips]{lscape}

\usepackage{tikz-cd}

\usepackage{hyperref} 

\usepackage[normalem]{ulem}

\begin{document}
\title{Height pairings for algebraic cycles   on\\
the product of  a curve and  a surface}
\author{Shou-Wu Zhang}
\maketitle

\theoremstyle{plain}
\newtheorem{thm}{Theorem}[section]
\newtheorem{theorem}[thm]{Theorem}
\newtheorem{proposition}[thm]{Proposition}
\newtheorem{prop}[thm]{Proposition}
\newtheorem{corollary}[thm]{Corollary}
\newtheorem{cor}[thm]{Corollary}
\newtheorem{lemma}[thm]{Lemma}
\newtheorem{lem}[thm]{Lemma}
\newtheorem{conjecture}[thm]{Conjecture}
\newtheorem{conj}[thm]{Conjecture}

\theoremstyle{Definition}
\newtheorem{definition}[thm]{Definition}
\newtheorem{defn}[thm]{Definition}

\theoremstyle{remark} 
\newtheorem{remark}[thm]{Remark}
\newtheorem{example}[thm]{Example}
\newtheorem{notation}[thm]{Notation}
\newtheorem{problem}[thm]{Problem}

\numberwithin{equation}{section}

 \newcommand{\sech}{ \mathrm{sech}\,}
 \newcommand{\csch}{\mathrm{csch}\,}

\newcommand{\BA}{{\mathbb {A}}}
\newcommand{\BB}{{\mathbb {B}}}
\newcommand{\BC}{{\mathbb {C}}}
\newcommand{\BD}{{\mathbb {D}}}
\newcommand{\BE}{{\mathbb {E}}}
\newcommand{\BF}{{\mathbb {F}}}
\newcommand{\BG}{{\mathbb {G}}}
\newcommand{\BH}{{\mathbb {H}}}
\newcommand{\BI}{{\mathbb {I}}}
\newcommand{\BJ}{{\mathbb {J}}}
\newcommand{\BK}{{\mathbb {K}}}
\newcommand{\BL}{{\mathbb {L}}}
\newcommand{\BM}{{\mathbb {M}}}
\newcommand{\BN}{{\mathbb {N}}}
\newcommand{\BO}{{\mathbb {O}}}
\newcommand{\BP}{{\mathbb {P}}}
\newcommand{\BQ}{{\mathbb {Q}}}
\newcommand{\BR}{{\mathbb {R}}}
\newcommand{\BS}{{\mathbb {S}}}
\newcommand{\BT}{{\mathbb {T}}}
\newcommand{\BU}{{\mathbb {U}}}
\newcommand{\BV}{{\mathbb {V}}}
\newcommand{\BW}{{\mathbb {W}}}
\newcommand{\BX}{{\mathbb {X}}}
\newcommand{\BY}{{\mathbb {Y}}}
\newcommand{\BZ}{{\mathbb {Z}}}

\newcommand{\CA}{{\mathcal {A}}}
\newcommand{\CB}{{\mathcal {B}}}
\newcommand{\CC}{{\mathcal{C}}}
\renewcommand{\CD}{{\mathcal{D}}}
\newcommand{\CE}{{\mathcal {E}}}
\newcommand{\CF}{{\mathcal {F}}}
\newcommand{\CG}{{\mathcal {G}}}
\newcommand{\CH}{{\mathcal {H}}}
\newcommand{\CI}{{\mathcal {I}}}
\newcommand{\CJ}{{\mathcal {J}}}
\newcommand{\CK}{{\mathcal {K}}}
\newcommand{\CL}{{\mathcal {L}}}
\newcommand{\CM}{{\mathcal {M}}}
\newcommand{\CN}{{\mathcal {N}}}
\newcommand{\CO}{{\mathcal {O}}}
\newcommand{\CP}{{\mathcal {P}}}
\newcommand{\CQ}{{\mathcal {Q}}}
\newcommand{\CR }{{\mathcal {R}}}
\newcommand{\CS}{{\mathcal {S}}}
\newcommand{\CT}{{\mathcal {T}}}
\newcommand{\CU}{{\mathcal {U}}}
\newcommand{\CV}{{\mathcal {V}}}
\newcommand{\CW}{{\mathcal {W}}}
\newcommand{\CX}{{\mathcal {X}}}
\newcommand{\CY}{{\mathcal {Y}}}
\newcommand{\CZ}{{\mathcal {Z}}}

\newcommand{\RA}{{\mathrm {A}}}
\newcommand{\RB}{{\mathrm {B}}}
\newcommand{\RC}{{\mathrm {C}}}
\newcommand{\RD}{{\mathrm {D}}}
\newcommand{\RE}{{\mathrm {E}}}
\newcommand{\RF}{{\mathrm {F}}}
\newcommand{\RG}{{\mathrm {G}}}
\newcommand{\RH}{{\mathrm {H}}}
\newcommand{\RI}{{\mathrm {I}}}
\newcommand{\RJ}{{\mathrm {J}}}
\newcommand{\RK}{{\mathrm {K}}}
\newcommand{\RL}{{\mathrm {L}}}
\newcommand{\RM}{{\mathrm {M}}}
\newcommand{\RN}{{\mathrm {N}}}
\newcommand{\RO}{{\mathrm {O}}}
\newcommand{\RP}{{\mathrm {P}}}
\newcommand{\RQ}{{\mathrm {Q}}}
\newcommand{\RS}{{\mathrm {S}}}
\newcommand{\RT}{{\mathrm {T}}}
\newcommand{\RU}{{\mathrm {U}}}
\newcommand{\RV}{{\mathrm {V}}}
\newcommand{\RW}{{\mathrm {W}}}
\newcommand{\RX}{{\mathrm {X}}}
\newcommand{\RY}{{\mathrm {Y}}}
\newcommand{\RZ}{{\mathrm {Z}}}

\newcommand{\fa}{{\frak a}}
\newcommand{\fg}{{\frak g}}
\newcommand{\fp}{{\frak p}}
\newcommand{\fk}{{\frak k}}
\newcommand{\fq}{{\frak q}}
\newcommand{\fl}{{\frak l}}
\newcommand{\fm}{{\frak m}}

\newcommand{\ab}{{\mathrm{ab}}}
\newcommand{\Ad}{{\mathrm{Ad}}}
\newcommand{\ad}{{\mathrm{ad}}}
\newcommand{\adm}{{\mathrm{adm}}}
\newcommand{\AJ}{{\mathrm{AJ}}}
\newcommand{\Alb }{{\mathrm{Alb}}}
\newcommand{\an}{{\mathrm{an}}}
\newcommand{\Aut}{{\mathrm{Aut}}}

\newcommand{\Br}{{\mathrm{Br}}}
\newcommand{\bs}{\backslash}
\newcommand{\bbs}{\|\cdot\|}

\newcommand{\Ch}{{\mathrm{Ch}}}
\newcommand{\cod}{{\mathrm{cod}}}
\newcommand{\coker}{{\mathrm{coker}}}
\newcommand{\cont}{{\mathrm{cont}}}
\newcommand{\cl}{{\mathrm{cl}}}
\newcommand{\cm}{{\mathrm{cm}}}

\newcommand{\der}{{\mathrm{der}}}
\newcommand{\dR}{{\mathrm{dR}}}
\newcommand{\disc}{{\mathrm{disc}}}
\newcommand{\Div}{{\mathrm{Div}}}
\renewcommand{\div}{{\mathrm{div}}}

\newcommand{\Eis}{{\mathrm{Eis}}}
\newcommand{\End}{{\mathrm{End}}}
\newcommand{\emb}{{\hookrightarrow}}

\newcommand{\fgl}{{\frak {gl}}}

\newcommand{\Frob}{{\mathrm{Frob}}}

\newcommand{\Gal}{{\mathrm{Gal}}}
\newcommand{\GL}{{\mathrm{GL}}}
\newcommand{\GO}{{\mathrm{GO}}}
\newcommand{\GSO}{{\mathrm{GSO}}}
\newcommand{\GSp}{{\mathrm{GSp}}}
\newcommand{\GSpin}{{\mathrm{GSpin}}}
\newcommand{\GU}{{\mathrm{GU}}}

\newcommand{\Hom}{{\mathrm{Hom}}}
\newcommand{\Hol}{{\mathrm{Hol}}}

\renewcommand{\Im}{{\mathrm{Im}}}
\newcommand{\Ind}{{\mathrm{Ind}}}
\newcommand{\inv}{{\mathrm{inv}}}
\newcommand{\Isom}{{\mathrm{Isom}}}

\newcommand{\Jac}{{\mathrm{Jac}}}
\newcommand{\JL}{{\mathrm{JL}}}

\newcommand{\Ker}{{\mathrm{Ker}}}
\newcommand{\KS}{{\mathrm{KS}}}

\newcommand{\Lie}{{\mathrm{Lie}}}

\newcommand{\new}{{\mathrm{new}}}
\newcommand{\NS}{{\mathrm{NS}}}
\newcommand{\NT}{{\mathrm{NT}}}

\newcommand{\ord}{{\mathrm{ord}}}
\newcommand{\ol}{\overline}

\newcommand{\rank}{{\mathrm{rank}}}
\newcommand{\rig}{{\mathrm{rig}}}

\newcommand{\perv}{\mathrm{perv}}
\newcommand{\PGL}{{\mathrm{PGL}}}
\newcommand{\PSL}{{\mathrm{PSL}}}
\newcommand{\Pic}{\mathrm{Pic}}
\newcommand{\Prep}{\mathrm{Prep}}
\newcommand{\Proj}{\mathrm{Proj}}
\newcommand{\PU}{\mathrm{PU}}

\renewcommand{\Re}{{\mathrm{Re}}}
\newcommand{\red}{{\mathrm{red}}}
\newcommand{\reg}{{\mathrm{reg}}}
\newcommand{\Res}{{\mathrm{Res}}}

\newcommand{\Sel}{{\mathrm{Sel}}}
\newcommand{\Sh}{{\mathrm{Sh}}}
\newcommand{\Shc}{{\mathcal {Sh}}}
\font\cyr=wncyr10  \newcommand{\Sha}{\hbox{\cyr X}}
\newcommand{\SL}{{\mathrm{SL}}}
\renewcommand{\sl}{{\mathfrak{sl}}}
\newcommand{\SO}{{\mathrm{SO}}}
\newcommand{\Sp}{\mathrm{Sp}}
\newcommand{\Spec}{{\mathrm{Spec}}}
\newcommand{\Sym}{{\mathrm{Sym}}}
\newcommand{\sgn}{{\mathrm{sgn}}}
\newcommand{\Supp}{{\mathrm{Supp}}}

\newcommand{\TC}{{\mathrm{TC}}}

\newcommand{\tr}{{\mathrm{tr}}}

\newcommand{\ur}{{\mathrm{ur}}}

\newcommand{\vol}{{\mathrm{vol}}}
\newcommand{\zar}{{\mathrm{zar}}}

\newcommand{\wt}{\widetilde}
\newcommand{\pp}{\frac{\partial\bar\partial}{\pi i}}
\newcommand{\ppr}{\frac{\partial\bar\partial}{2\pi i}}
\newcommand{\intn}[1]{\left( {#1} \right)}
\newcommand{\norm}[1]{\|{#1}\|}
\newcommand{\sfrac}[2]{\left( \frac {#1}{#2}\right)}
\newcommand{\ds}{\displaystyle}
\newcommand{\ov}{\overline}
\newcommand{\incl}{\hookrightarrow}
\newcommand{\imp}{\Longrightarrow}
\newcommand{\lto}{\longmapsto}
\newcommand{\sL}{{\mathsf L}}
\newcommand{\surj}{\twoheadrightarrow}

\newcommand{\pair}[1]{\langle {#1} \rangle}
\newcommand{\prim}{\mathrm{prim}}
\newcommand{\wpair}[1]{\left\{{#1}\right\}}
\newcommand\wh{\widehat}
\newcommand\Spf{\mathrm{Spf}}
\newcommand{\lra}{{\longrightarrow}}
\newcommand{\iso}{{\overset\sim\lra}}

\newcommand{\Ei}{\mathrm{Ei}} 

\newcommand{\todo}[1]{{\color{blue}
    \textsf{[#1]}}}

\begin{abstract}
For the product $X=C\times S$ of a curve and a surface over a number field,  
we prove  Beilinson--Bloch's conjecture about the existence of a height pairing  (\cite{Be, Bl}) between homologically trivial cycles. Then, for an embedding $f: C\lra S$, we construct an arithmetic diagonal cycle modified 
from the graph of $f$ and study its height.
This work extends the previous work of Gross and Schoen \cite{GS95}
to $X$  the product of three curves, and makes Gan--Gross--Prasad conjecture unconditional
for $\RO(1,2)\times \RO(2,2)$ and $\RU(1,1)\times \RU(2,1)$.
\end{abstract}

\tableofcontents

\section{Introduction}\label{sec-int}
Let $K$ be a number field or the function field of a curve $B$ over a field $k$, and
let $X=C\times S$ be the product of a curve and a surface, both smooth and projective over $K$.
This paper aims to extend the main results of Gross and Schoen \cite{GS95} in the case $S$ is the product of two curves. There are two goals.

Our first goal is to prove Beilinson--Bloch's conjecture about the existence of the height pairing:
$$\pair{\cdot, \cdot}_\BB: \Ch^2(X)^0\otimes \Ch^2(X)^0\lra \BR,$$
where $\Ch^2(X)^0$ denote the Chow group of homologically trivial cycles of codimension $2$.
We recover the result of Gross and Schoen for triple products of curves with a new proof. Gross and Schoen used Tate's conjectures, which are known for triple products of curves. Furthermore, our proof extends their result to the case of products of curves and surfaces, for which Tate's conjecture is open but for which Grothendieck's standard conjectures can be shown to be held and imply the existence of Beilinson--Bloch height pairings by our recent work \cite{Zh21}.

Our second goal is to construct homologically trivial cycles  $\gamma\in \Ch^2(X)^0$ modified from the graph of finite morphisms $f: C\lra  S$. 
 As an attempt to generalize our results 
 in \cite{Zh10} from triple products of curves to products of curves and surfaces, we will provide a formula for the height $\pair{\gamma, \gamma}_\BB$ in the case that $K$ is a function field $K=k(B)$ 
 and where both $C$ and $S$ have smooth models over $B$. When $S$ is a K3 surface in the function field case with good reduction, our formula shows that   Beilinson's Hodge index conjecture \cite{Be} gives an inequality between the $\omega_C^2$ and the canonical height of  $S$.   We expect a precise formula to hold in general for  
 function field and number field situations and have interesting applications to Diophantine problems.

 This work arose from an attempt to understand a   conjecture proposed by   Gan, Gross,  and Prasad (GGP) \cite{GGP, Zh18, Zh19}, which relates the derivatives of  $L$-series to the conjectured height pairings of arithmetic diagonal cycles for embeddings of Shimura varieties. Our work provides such height pairings and thus makes the GGP conjecture unconditional for $\RO(1,2)\times \RO(2,2)$ and $\RU(1,1)\times \RU (2,1)$. More precisely, 
 when $f$ embeds a Shimura curve into a Shimura surface, then for each Hecke operator $t$ of $X$, 
 we have a cycle $t_*\gamma\in \Ch^2(X)^0$. Moreover, when $t$ is an idempotent for some new form in some automorphic representation $\pi$, then  $\pair{t_*\gamma, t_*\gamma}_\BB$ is essentially equal to the central derivative of the specific $L$-series $L(s, \pi)$.

 The plan of this paper is as follows. In \S2-3 for any  given  polarizations $\eta$ of $C$ and $\xi$ of $S$,
  we will  first give some canonical decompositions   for the groups $A^*(X)$ of algebraic cohomology cycles and $\Ch^*(X)^0$ of homologically trivial cycles. 
  There are  two consequences of these decompositions: one is   Grothendieck's standard conjecture for $A^*(\wt X)$ 
  for any $\wt X$ obtained from $X$ by 
  successive blow-ups along smooth curves and points, another is a decomposition $\Ch^*(X)=\Ch^*(X)^0\oplus A^*(X)$.
  Then we define certain quotients $J^i(X)$ of $\Ch^*(X)^0$, which we call the intermediate Jacobians and the bi-primitive subgroup $J^2(X)_{00}$ annihilated by both $\xi$ and $\eta$.
  Finally, we provide a canonical way to modify a cycle in $\Ch^*(X)$ to become a cycle in  $J^2(X)_{00}$. 
  
  In \S4-5,  when  both $C$ and $S$ have strictly semistable models $\CC$ and $\CS$ over a discrete valuation ring $R$, 
  we will first construct a strictly semistable model $\CX$ for $X$ by successive blowing up along the strict transforms of the 
  irreducible components in the special fiber of $\CC\times_R \CS$. 
  The results  in section \S2-3  will imply that the resulting irreducible components in the special fiber of $\CX$
  satisfy the Grothendieck standard conjectures. Then  we apply  our recent  work in  \cite{Zh21} to define the Beilinson--Bloch height pairings 
  $\pair{\cdot, \cdot}_\BB$ for 
  $\Ch^*(X)^0$ when $X$ is defined over number  fields or function fields.
  Finally, we extend the height pairing to general $C$ and $S$ using de Jong's alteration \cite{dJ}.
   
  In \S6-7, first, we define the arithmetic diagonal  $\gamma$ as mentioned above. 
  Then, we restrict our study in the case $H^1(S)=0$ and $\deg f^*f_*C\ne 0$. In this case, we have a simpler formula describing the cycle $\gamma$.  Finally,  we give a  formula for $\pair{\gamma, \gamma}_\BB$ in the function field case $K=k(B)$
  when both $C$ and $S$ have smooth models over $B$. 
  \\
  
 \noindent{\bf Acknowledgments.}  
 We want to thank Benedict Gross, Klaus K\"unnemann, Yifeng Liu, Xinyi Yuan, 
 and Wei Zhang for their comments on an early draft of this paper. We also thank the referee for the valuable comments 
 and suggestions.
 This work has been partially supported by an NSF award, DMS-2101787.

\section{Algebraic  cohomology cycles}\label{acc}
Let $X$ be a smooth and projective variety over an algebraically closed field $K$ of dimension $n$.
For an integer $i$, let  $\Ch^i(X)$ denote the group of Chow cycles of codimension $i$ 
with rational coefficients modulo rational equivalence, and consider the cohomological class map:
$$\cl^i: \Ch^i(X)\lra H^{2i}(X)(i)$$
in a fixed Weil cohomology theory. 
Let $A^i(X)$ denote the image of $\cl^i$ and let $\Ch^i(X)^0$ denote the kernel of $\cl^i$.

When $i=1$ or $i=n$, these groups are related to abelian varieties by  the identity $\Ch^1(X)^0=\Pic^0(X)_\BQ$  and the 
Abel--Jacobi map 
$$\AJ: \Ch^n(X)^0\lra \Alb(X)(K)\otimes _\BZ\BQ.$$

Let $\CL$ be an ample line bundle on $X$. 
Then taking the intersection with $c_1(\CL)$ defines a Lefschetz operator $\sL: A^i(X)\lra A^{i+1}(X)$.
Here is the Grothendieck standard conjecture.
\begin{conj}[Grothendieck] Let $i\le n/2$. 
\begin{enumerate}
\item Hard Lefschetz theorem:  The following induced map is bijective: 
$$\sL^{n-2i}: A^i(X)\iso A^{n-i}(X).$$
\item Hodge index theorem: For a nonzero $\alpha\in A^i(X)$ such that $\sL^{n+1-2i}\alpha=0$, 
$$(-1)^i\alpha \cdot \sL ^{n-2i}\alpha >0.$$
\end{enumerate}
\end{conj}

\begin{remark}
For $i=0$, the conjecture is trivially true. For $i=1$,  the Hodge index theorem is true; see \cite[Expos\'e XIII, Corollary 7.4]{SGA}. This implies that $\sL^{n-2}: A^1(X)\lra A^{n-1}(X)$ is always injective and that $\dim A^1(X)<\infty$.
Thus, the conjecture is true for $n=1, 2$, and is equivalent to the inequality $\dim A^1(X)\ge \dim A^2(X)$ for $n=3$.
\end{remark}

In the rest of this section, we will assume that $X$ is of the form $C\times S$ with $C$ a curve and $S$ a surface. Let $\pi_C$ and $\pi_S$ denote two projections. We will consider $\Ch^*(X)$ as the group  correspondences between $\Ch^*(C)$ and $\Ch^*(S)$ as usual: for any $\alpha \in \Ch^*(X)$ and $\beta \in \Ch^*(C)$ and $\gamma \in \Ch^*(S)$, we define 
$$\alpha _*(\beta)=\pi_{S*}(\alpha \cdot \pi_C^*\beta)\in \Ch^*(S), 
\qquad \alpha ^*(\gamma)=\pi_{C*} (\alpha\cdot \pi_S^*\gamma) \in \Ch^*(C).$$
For  classes $\alpha \in \Ch^i(C)$ and $\beta \in \Ch^j(S)$,  we denote 
$$\alpha \boxtimes \beta=\pi_C^*\alpha \cdot \pi_S^* \beta\in \Ch^{i+j}(X).$$
We will use the same notation for the induced operation in the K\"unneth decomposition:
$$\boxtimes: A^*(C)\otimes A^*(S)\lra A^*(X).$$

We will fix  ample classes $\xi\in \Ch^1(S)$, $\eta\in \Ch^1(C)$  such that $\deg \eta=1$. We consider them as cycles on $X$ via pull-backs. 
Let $A^i(X)_{\rig}$ denote the subgroup  of $A^i(X)$ of elements $\alpha$ with rigidifications:
$$\alpha _*(\eta)=0, \qquad \alpha _*([C])=0.$$
It can be checked that  $A^i(X)_{\rig}$ is the intersection of $A^i(X)$ and  $H^1(C)\otimes H^{2i-1}(S)$ in the K\"unneth decomposition:
$$H^{2i}(X)(i)=\bigoplus _{j=0}^{2}H^j(C)\boxtimes H^{2i-j}(S)(i).$$
Our starting point is a  description of the group $A^*(X)$ in terms of $\xi, \eta, A^*(C), A^*(S)$, and $A^*(X)_{\rig}$.
\begin{prop}\label{prop-A} Let $X=C\times S$ and $\xi$ and $\eta$ be as above. Then
\begin{enumerate}
\item  $A^0(X)=\BQ [X]$, $A^3(X)=\BQ (\eta\boxtimes \xi^2)$.
\item For $i=1,2$,  $ A^i(X)=A^i(X)_{\rig}\oplus [C]\boxtimes A^i(S)\oplus \eta\boxtimes  A^{i-1}(S)$ with respect to the decomposition 
$$H^{2i}(X)=H^1(C)\boxtimes H^{2i-1}(S) \oplus [C]\boxtimes H^{2i}(S)\oplus \eta\boxtimes H^{2i-2}(S).$$
The projection to the last two components is given by 
\begin{equation}\label{eq-kd}
\alpha\mapsto [C]\boxtimes  \alpha_*(\eta)+\eta\boxtimes \alpha_*[C].\end{equation}
\item The intersection with $\xi$ gives a bijection $A^1(X)_{\rig}\iso A^2(X)_{\rig}$ with respect to the isomorphism
given by intersection with $\pi_S^*\xi$:
$$H^1(C)\boxtimes H^1(S)\iso H^1(C)\boxtimes H^3(S).$$
\end{enumerate}
\end{prop}
\begin{proof} The first part is clear. For the second part, we notice that the map \ref{eq-kd} takes cycles in $A^i(X)$ to cycles in $A^i(X)$. Thus it induces a decomposition on $A^i(X)$. For the third part, we only need to prove the surjectivity. Let $\alpha \in A^2(X)_{\rig}$. 
We lift $\alpha$ to a class  $\wt\alpha\in \Ch^2(X)$ and consider the morphism of abelian varieties:
$$\wt\alpha^*: \Pic^0 (S)\lra \Pic^0(C).$$
This map does not depend on the choice of lifting $\wt\alpha$ because $\alpha^*$ induces the following map on  cohomology:
$$H^1(\Pic^0(C))=H^1(C)\overset {\alpha ^*} \lra H^3(S)(1)=H^1(\Pic ^0(S)).$$
Now we combine the map $\wt\alpha ^*$  with the following polarization map defined by $\xi$:
$$\varphi: \Pic^0(S)\lra \Alb (S): \CM\mapsto \AJ(\CM\cdot \xi),$$
and  obtain an  $f:=\wt\alpha ^*\cdot \varphi^{-1}\in \Hom (\Alb (S), \Pic ^0(C))_\BQ$.
This  $f$  is represented by a unique  class $\mu\in \Ch^1(X)$ such that $\mu_*(\eta)=0$ and $\mu ^*(\xi^2)=0$. 
From construction, it is clear that 
$$\alpha =c_1(\mu)\cdot \xi.$$
 \end{proof}
 
 \begin{remark}[Functoriality and correspondences] Let 
 $$f=f_{C'}\times f_{S'}: X':=C'\times S'\lra C\times S$$ be a finite morphism with $C'$ and $S'$ smooth and projective. Then, we have the induced  homomorphisms of groups:
 $$f_*: A^*(X')\lra A^*(X), \qquad f^*: A^*(X)\lra A^*(X').$$
 Let $\xi', \eta'$ be polarizations of $S'$ and $C'$ with 
 $\xi'\in \BQ f^*\xi$ and $\eta'\in \BQ f^*\eta$. 
 These two maps respect the decomposition in Proposition \ref{prop-A}. More generally,  we have two actions by correspondence: 
  $$A^3(X\times X)\lra \End (A^*(X)):  \qquad \alpha \mapsto \alpha_*, \quad \alpha \mapsto \alpha ^*.$$
  One is a homomorphism of algebras over $\BQ$, and another is an anti-homomorphism:
  $$(\alpha_1\circ \alpha_2)^*=\alpha_1^*\circ \alpha_2^*, \qquad (\alpha _1\circ \alpha _2)_*=\alpha_{1*}\circ \alpha _{2*}.$$
  Let $A^1(C\times C)_\rig$ (resp. $A^2(S\times S)_\rig$) be the subgroup of $A^1(C\times C)$ (resp. $A^2 (S\times S)$)
  of elements $\alpha$ such that both $\alpha _*$ and $\alpha ^*$ fix the line $\BQ\eta$ (resp. $\BQ \xi$).
  Define $$A^3(X\times X)_\rig:=A^1(C\times C)_\rig \boxtimes A^2(S\times S)_\rig .$$
Then, the above actions restricted to  $A^3(X\times X)_\rig$ will respect the decomposition in Proposition \ref{prop-A}.
  \end{remark}

By Proposition \ref{prop-A},  $A^2(X)$ and $A^1(X)$ have the same dimension. Thus, we have the following Corollary.

\begin{cor} Let $X=C\times S$ be the product of a curve and a surface. Then Grothendieck's standard conjecture holds.
\end{cor}

We also need the standard conjectures for blow-ups of $X$ to define Beilinson--Bloch height pairings in \S5:

\begin{prop} Let $Y$ be a smooth and projective threefold.
 Let $\pi: \wt Y\lra Y$ be a blow-up of $Y$ along a smooth subvariety $Z$.
 Then, the standard conjectures for $Y$ and $\wt Y$ are equivalent.  
 \end{prop}
 
 \begin{proof} Let $\iota: E\lra \wt Y$ be the exceptional divisor. Then $A^*(\wt Y)$ is generated by $\pi^*A^*(Y)$ and  $\iota_*A^*(E)$.
 Notice that  $E$ is a projective bundle over $Z$. So $A^*(E)$ is generated over $A^*(Z)$ by the first Chern class of $\CO_E(1)$.
 Thus we have $A^1(\wt Y)=\pi^*A^1(Y)+\BQ [E]$, and $A^2(\wt Y)=\pi^*A^2(Y)+\BQ E^2$. This shows that $\dim A^1(Y)=\dim A^2(Y)$ is 
 equivalent to that $\dim A^1(\wt Y)=\dim A^2(\wt Y)$. 
 See \cite[\S6.7]{Fu} for more details.
 \end{proof}
 
 \begin{cor}\label{cor-sc}  Let $\wt X$ be a smooth and projective three-fold obtained from 
 the product $X=C\times S$ of curve and surface by successively blowing up along smooth curves.
 Then, the standard conjecture holds for $\wt X$.  \end{cor}

 \section{Homologically trivial cycles}\label{sec-htc}
  In this section, let $K$ be an algebraically closed field and $X=C\times S$ be the product of a curve and a surface, both smooth and projective over $K$.
  We will fix  classes   $\xi\in \Ch^1(S)$,   $\eta\in \Ch^1(C)$ such that $\xi$ is ample and 
$\deg \eta=1$. Then we can define analogously the subgroup $\Ch^1(X)_{\rig}$ of  $\Ch^1(X)$ of elements  with following rigidifications:
 $$\alpha_*(\eta)=0, \qquad \alpha ^*(\xi^2)=0.$$ Then the map $\Ch^1(X)\lra A^1(X)$ defines a bijection 
 $\Ch^1(X)_{\rig}\lra A^1(X)_{\rig}$. Thus Proposition \ref{prop-A} defines maps for $i=1,2$:
 $\alpha\mapsto \alpha_\mu: \Ch^i(X)\lra \Ch^1(X)_{\rig}$.
 
 The combination of parts 2 and 3 of Proposition \ref{prop-A} gives the following
 \begin{prop}\label{prop-Ch}  For  $i=1,2$, we have a decomposition:
  $$\Ch^i(X)=\Ch^i(X)^0\oplus [C]\boxtimes  \Ch^i(S)\oplus \eta\boxtimes \Ch^{i-1}(S)\oplus\Ch^1(X)_{\rig}\xi^{i-1}.$$
 such that the projection to the last three components is given by 
 $$\alpha\mapsto [C]\boxtimes \alpha_*(\eta)+\eta\boxtimes \alpha _*[C]+\alpha _\mu\xi^{i-1}.$$
  \end{prop}
  
  \begin{remark}[Functoriality and correspondences]\label{rem-fc} Let $$f=f_{C'}\times f_{S'}: X':=C'\times S'\lra C\times S$$ be a generically finite morphism with $C'$ and $S'$ smooth and projective.  Then we have group homomorphisms
 $$f_*: \Ch^*(X')\lra \Ch^*(X), \qquad f^*: \Ch^*(X)\lra \Ch^*(X').$$
 Let $\xi', \eta'$ be polarizations of $S'$ and $C'$ with 
 $\xi'\in \BQ f^*\xi$ and $\eta'=\BQ f^*\eta$. Then, the above homomorphisms respect the decomposition in Proposition \ref{prop-Ch}.
 Moreover, we have two  actions  by the correspondence $\alpha$: 
  $$\Ch^3(X\times X)\lra \End (\Ch^*(X)):  \qquad \alpha \mapsto \alpha_*, \quad \alpha \mapsto \alpha ^*.$$
  One is a $\BQ$-algebra  homomorphism; the other is a $\BQ$-algebra  anti-homomorphism.
  Let $\Ch^1(C\times C)_\rig$ (resp. $\Ch^2(S\times S)_\rig$) be the subgroup of $\Ch^1(C\times C)$ (resp. $\Ch^2 (S\times S)$)
  of elements $\alpha$ such that both $\alpha _*$ and $\alpha ^*$ fix the line $\BQ\eta$ (resp. $\BQ \xi$).
  Let  $\Ch^3(X\times X)_\rig$ denote the product $\Ch^1(C\times C)_\rig \boxtimes \Ch^2(S\times S)_\rig $. 
Then, the above actions restricted to  $\Ch^3(X\times X)_\rig$ will respect the decomposition in Proposition \ref{prop-Ch}.
 \end{remark}

We want to define some ``intermediate   Abel--Jacobi maps"
$\AJ ^i: \Ch^i(X)^0\lra J^i(X)$ with actions by $A^*(X)$.
For $i=1$, we take $J^1(X)=\Pic ^0(X)(K)_\BQ$;  for $i=3$, we take 
$J^3(X):=\Alb (X)(K)_\BQ$; 
for $i=2$, we  define the  modified group by 
$$J^2(X)=
\Ch^2(X)^0/\sum_f f_*(\Ch^1(X')^0\cdot \Ch^1(X')^0).$$
where $f$ runs  generically finite morphisms 
\begin{equation}
\label{gf}
f=f_{C'}\times f_{S'}: X'=C'\times S'\lra X=C\times S
\end{equation} with 
$C'$ and $S'$ smooth and projective curves and surfaces, respectively.
Notice that for an embedding $K\subset \BC$, when $i=2$, there is a Griffiths' intermediate Jacobian defined by Hodge theory:
$$J^i_G(X_\BC):=F^iH^{2i-1}(X(\BC), \BC)\bs H^{2i-1}(X(\BC), \BC)/H^{2i-1}(X(\BC), \BZ).$$
There will be a map. 
$$J^i(X_\BC)\lra J^i_G(X(\BC))\otimes \BQ.$$
By definition, this is an isomorphism when $i=1$ or $3$. When $i=2$ and $K=\bar\BQ$,
 by a conjecture of Beilinson \cite[p.18]{Be}, the restriction on $J^2(X)$ of this map  is injective.

\begin{prop} 
The intersection pairing on $\Ch^*(X)$ induces an action:
$$A^i(X)\otimes J^j(X)\lra J^{i+j}(X).$$
\end{prop}

\begin{proof} This follows immediately from the following Lemma.
\end{proof}

\begin{lem}\label{lem-JCh}
 Let $\alpha\in \Ch^1(X)^0$ and $\beta \in \Ch^2(X)^0$. Then $\AJ (\alpha\beta)=0$.
\end{lem}
\begin{proof}
For any $\beta\in \Ch^2(X)$, the map  $\alpha\mapsto \AJ (\alpha\beta)$ defines a morphism of abelian groups 
$\varphi_\beta: \Pic ^0(X)(K)\lra \Alb (X)(K)$. We claim that a morphism of abelian varieties induces this.
Using duality, $\Alb (X)=\Pic ^0(\Pic ^0(X))$,  we need to construct a line bundle $\CM$ on $\Pic ^0(X)\times \Pic ^0(X)$  so that $\varphi_\beta (\alpha)$ is given the restriction of $\CM$ on 
$\{\alpha\}\times \Pic^0(X)$. 

Fixed base point $x\in X$, then we have a universal bundle 
$\CP$ on $X\times \Pic ^0(X)$ with trivializations on $X\times \{0\}$ and $\{x\}\times \Pic ^0(X)$ so that for any $\alpha \in\Pic ^0(X)$, the restriction of $\CP$ on $X\times \times \{\alpha\}$ is a line bundle representing $\alpha$. 
Then consider variety $X\times \Pic ^0(X)\times \Pic ^0(X)$  and cyles $p_{12}^*c_1(\CP)$, $p_{13}^*c_1(\CP)$ and $p_1^*\beta$ defined by projections to the products $X\times \Pic ^0(X)$ and $X$ respectively. 
So we get a class
$$c_1(\CM):=p_{23*} (p_{12}^*c_1(\CP)\cdot p_{13}^*c_1(\CP)\cdot p_1^*\beta)\in \Ch^1(\Pic ^0(X)\times \Pic^0(X)).$$
This is the class inducing $\varphi_\beta$. 

This class of $\CM$ certainly has a trivial restriction on $\{0\}\times \Pic^0(X)$ and 
$\Pic^0(X)\otimes \{0\}$. Thus, it is determined by its cohomological class in $H^2(\Pic ^0(X)\times \Pic^0(X))$ which can be computed by the same formula 
$$\cl ^1c_1(\CM)=p_{23*} (p_{12}^*\cl^1c_1(\CP)\cdot \cl^1 p_{13}^*c_1(\CP)\cdot p_1^*\cl^2\beta).$$
Thus $\CM$ is trvial if $\cl^2\beta=0$. Thus $\beta \in \Ch^2(X)^0$ implies that $\CM=0$ and that $\varphi_\beta=0$.
\end{proof}

 \begin{remark} Let $f: X'\lra X$ be as in (\ref{gf}). Then we have morphisms
 $$f_*: J^*(X')\lra J^*(X), \qquad f^*: J^*(X)\lra J^*(X').$$
 More generally,  we have two actions by correspondence: 
  $$\Ch^3(X\times X)\lra \End (J^*(X)):  \qquad \alpha \mapsto \alpha_*, \quad \alpha \mapsto \alpha ^*.$$
  We don't know if these actions factor through the action by   $A^*(X\times X)$ in general. However, if $K=\bar \BQ$, then
  by  Beilinson's conjecture, $J^i(X)$ can embed into the Griffith intermediate Jacobian $J_G^i(X(\BC))\otimes \BQ$,
  Thus, the action of $\Ch^3(X\times X)$ conjecturally factors through action by   $A^3(X\times X)$ in this case.
  \end{remark}

Write $\CL=\xi\boxtimes\eta$ as an ample line bundle on $X$. The intersection with $c_1(\CL)$ then defines an operator 
$$\sL: J^i(X)\lra J^{i+1}(X).$$
Notice that $\sL^2: J^1(X)\lra J^3(X)$ is a polarization. 
For $J^2(X)$, let $J^2(X)_0$ denote the kernel of $\sL$ called the primitive part of $J^2(X)$. Then, we have a decomposition:
$$J^2(X)=\sL J^1(X)\oplus J^2(X)_0.$$

For simplicity, we will  work with  the  subgroup  $J^2(X)_{00}$ of $J^2(X)_{0}$ of elements annihilated by both $\pi_C^*\eta$ and $\pi_S^*\xi$. 
We call this the group of the bi-primitive cycles. 
These conditions are  easy to describe by noticing that 
$$J^3(X)\iso \Alb (C)\times \Alb (S): \qquad \alpha\mapsto (\pi_{C_*}\alpha, \pi_{S*}\alpha).$$
Thus an element $\alpha \in J^2(X)$ is in $J^2(X)_{00}$ if and only if 
$$\alpha _*(\eta) =0, \quad \eta\cdot \alpha^*S=0, \quad \alpha ^*(\xi)=0, \quad \alpha _*(C)\cdot \xi=0.$$
Since $\alpha^*S=0$ always true, and $\alpha_*(C)\xi=0$ is equivalent to $\alpha_*(C)=0$, we need only check the following three conditions:
$$\alpha ^*(\xi)=0, \qquad \alpha _*(\eta) =0,   \quad \alpha _*(C)=0.$$
Thus, we have proved the following:
\begin{prop} \label{prop-J} Let $\xi^\vee=\xi/(\deg \xi^2)$. Then, we have a decomposition:
$$J^2(X)=J^2(X)_{00} \oplus [C]\boxtimes \Alb(S)\oplus \eta \boxtimes \Pic ^0(S)\oplus \Pic^0 (C)\boxtimes \xi^\vee,$$
such that the projection to the last three components is given by 
$$\alpha\mapsto [C]\boxtimes \alpha_*(\eta)+\eta\boxtimes \alpha_*([C])+\alpha^*(\xi)\boxtimes \xi^\vee.$$
\end{prop}

\begin{remark}
 Let $f: X'\lra X$ and  $\xi', \eta'$ be as in Remark \ref{rem-fc}. Then the maps 
  $$f_*: J^*(X')\lra J^*(X), \qquad f^*: J^*(X)\lra J^*(X')$$
 respect to the above decomposition. So do the actions by correspondences
  in   $\Ch^3(X\times X)_\rig$ defined in Remark \ref{rem-fc}. 
 \end{remark}

Combining with Proposition \ref{prop-A}, we obtain the following:
\begin{cor} \label{cor-ChJ}
The inclusion $$J^2(X)_{00}\lra \Ch^2(X)/\sum _f f_*(\Ch^1(X')^0\cdot \Ch^1(X')^0)$$ has a retraction 
given by
$$\alpha\mapsto \alpha-[C]\boxtimes \alpha_*(\eta)-\eta\boxtimes \alpha_*([C])-\mu_\alpha \xi-(\alpha^*(\xi)-\deg \alpha^*\xi\cdot \eta)\boxtimes \xi^\vee.$$
\end{cor}

Let $\NS(S)_0$ denote the orthogonal complement of $\xi$ in $\NS(S)$ under the intersection pairing $\pair{\cdot, \cdot}_\NS$. The  $J^2(X)_{00}$ can be decomposed into the subspace $J^2(X)_{000}$ annihilated by all elements in $\pi_S^*\NS(S)_0$. Then, we have a further decomposition.

\begin{prop}\label{prop-000} There is a decomposition 
$$J^2(X)_{00}=J^2(X)_{000}\oplus \Pic^0(C)\boxtimes \NS(S)_0$$
so that the projection to the second component is given as follows:
$$\alpha \mapsto \sum _i \alpha ^*(h_i)\boxtimes h_i^\vee $$
where  $\{h_1, \cdots h_t\}$ is a base of $\NS(S)_0$, and  $\{h_1^\vee, \cdots, h_t^\vee\}$ its dual base
in the sense that $\pair{h_i, h_j^\vee }_\NS=\delta _{ij}$.
\end{prop}

\section{Integral models}\label{sec-it}
In this section, we consider the local situation.
Let $K$ be a discrete valuation field with the valuation ring $R=\CO_K$,  a uniformizer $\pi$, and the residue field $k=R/\pi$.
For an $R$-variety $\CV$ and a closed point $P\in \CV_k$, we say that $\CV$ is  strictly  semistable  at $P$, if the the formal 
 completion of $\CV$ at $P$ admits a   finite \'etale map to the following formal scheme:
$$\Spf R[[x_1, \cdots, x_n]]/(\prod _{i=1}^t x_i-\pi)$$
for some $t\le n$. We say that $\CV$ is strictly semistable if it is strictly semistable at each close point. 
The property of strict semistability does not change if we replace $R$ with an unramified extension. So, in the following,
we assume that $R$ is complete and $k$ is algebraically closed. In this case, the above \'etale map is an isomorphism.

Let $X=C\times S$ be the product of a smooth proper curve and a proper smooth surface over $K$.
Assume that $C$ (resp.  $S$) has a  strictly semistable model $\CC$ (resp. $\CS$) over $R$.
Following Gross and Schoen \cite{GS95} (see also Hartl \cite{Ha}), we have  a semistable model $\CX$ for $X$ by blowing-up $\CX_1:=\CC\times _R\CS$.
More precisely, take any order of the irreducible components of $\CX_{1, k}$:
$C_1^1, \cdots, C_1^t$. Then we define models $\CX_i$ for $i=1, \cdots, t$  as follows:
$\CX_2$ is the blow-up of $\CX_1$ over $C_1$. Then the special fiber of $\CX_2$ is the union of the preimage of  $C_2^i$  of $C_1^i$.
Take $\CX_3$ be the blow-up of $\CX_2$ over $C_2^2$, etc. 

Our main result in this section (Proposition \ref{prop-cb}) is a description of $C_t^i$ in terms of blow-ups from $C_1^i$ without reference to $\CX$.
We need to review the construction of $\CX$ to prove this result. 

\begin{prop}[Hartl \cite{Ha}] \label{prop-ss}The scheme $\CX_t$ is strictly semistable. 
\end{prop}
\begin{proof} This is a very special case of a general result by Hartl  \cite[9, Proposition 2.1]{Ha}.
For our application, the proof of the next Proposition, we repeat the proof in our case. 
Let $P$ be a closed point of $\CX_t$. We need to show that the completion of $\CX_t$ at $P$ 
has the form
$$\Spf R[[t_1, \cdots, t_4]]/(\prod _{i=1}^q t_i-\pi)$$
for some integer $q$ between $1$ and $4$.
Let $P_i$ be the image of $P$ on $\CX_i$, and let $\wh \CX_i=\Spf \CO_i$ be the completion of $\CX_i$ at $P_i$. 

 Write $P_1=(p, q)$ with $p\in \CC_k$ and $q\in \CS_k$.
 If $p$ or $q$ is a smooth point, then $\CX_1$ is strictly semistable at $P_1$. Thus near $P_1$, all irreducible components of $\CX_{k}$ are Cartier.
 So, a blow-up along such a component does change the local completion at $P_1$. In this case, all $\CO_i$ are isomorphic to each other. 
 So, we may assume that both $p$ and $q$  are singular in the following discussion. In this case, $p$ is the intersection of two components $A_1\cdot A_2$,
and $q$ is the intersection of two or three components $B_j$. Then, the complete local ring of $\CO_1$ is given as one of the following two cases:
$$\CO_1=\begin{cases}R[[x_1, x_2, y_1, y_2, y_3]]/(y_1y_2-\pi, x_1x_2-y_1y_2),&\text{or}\\
R[[x_1, x_2, y_1, y_2, y_3]]/(y_1y_2y_3-\pi, x_1x_2-y_1y_2y_3)
\end{cases}
$$
Modulo the uniformizer $\pi$,  $\CX_{1 k}$ has  4 or 6 irreducible components $C_{i, j}$ defined by 
ideals $(x_i, y_j)$ for $i=1, 2$,  and $j=1,2$, or $3$. They are formal completions at $P_1$ of irreducible components $C_1^{i_k}$ for 
$1\le i_1<  \cdots < i_{s}\le t$, where $s=4$ or $6$.
It follows that $\CO_i=\CO_1$ for $i\le i_1$. Without loss of generality, we assume that $C_{1,1}=C_1^{i_1}$. 
Then $P_{i_1+1}$ is on the blowing up $\Spf \CO_{i_1}$ by  the ideal $(x_1, y_1)$. The formal completion of two affines covers this formal scheme with formal rings in two cases
$$\begin{cases} 
R[x_1', x_2, y_1,  y_3]/(x_1'x_2y_1-\pi),\quad R[x_1,  y_1', y_2, y_3]/(x_1y_1'y_2-\pi), \quad \text{or}\\
R[x_1', x_2, y_1, y_2, y_3]/(y_1y_2y_3-\pi, x_1'x_2-y_2y_3), \quad R[x_1,  y_1', y_2, y_3]/(x_1y_1'y_2y_3-\pi)
\end{cases}
$$
These are $\CO_1$-algebra generated by $x_1'=x_1/y_1, y_1'=y_1/x_1$. The point $P_{i_1+1}$ is defined by the ideal of $P_{i_1}$ and an element $z$  equal to 
$x_1'$ or $y_1'-a$ for some $a\in R$.  The above rings are strictly semitable at $P_{i_1}$  except  one case:
$$R[x_1', x_2, y_1, y_2, y_3]/(y_1y_2y_3-\pi, x_1'x_2-y_2y_3).$$
In this case, we have 
$$
\CO_{i_1+1} =R[[x_1', x_2, y_1, y_2, y_3]]/(y_1y_2y_3-\pi, x_1'x_2-y_2y_3), \qquad z=x_1'.$$
 One more blow-up will make $P_t$ a strictly semistable point for $\CX_t$.
\end{proof}

Next, we want to show that  $C_t^i$ can be obtained from $C_1^i$ by several blow-ups along smooth curves.
Let $Z^{i, j}=C_1^i\cap C_1^j$ for $i\ne j$. Then $Z^{i, j}$ has the form $(A_1\cap A_2)\times (B_1\cap B_2)$.
It is either  a smooth curve (when $A_1\ne A_2$, $B_1\ne B_2$), or a cartier divisor (when   $A_1=A_2$ or $B_1=B_2$). 
\begin{prop} \label{prop-cb}
For each $i$ between $1$ and $t$, the projection $C_t^i\lra C_1^i$ can be obtained by blowing up successively along  the strict transforms of the following smooth curves:
$$Z^{1, i}, \cdots, Z^{i-1, i}, Z^{i+1, i}, \cdots Z^{t, i}$$
\end{prop}
\begin{proof}
For each $j$ between  $1$ and $t$ and different than $i$,
let $D_j^i$ denote the variety by successively blowing-up from $C_1^i$ along the strict transforms of $Z^{i, j-1}$'s if $j\ne i+1$.
We let $D_{i+1}^i=D_i^i$. 
The construction of $C_j^i$ is the same except for one extra blowing up on $\CX_i$ along $C_i^i$. Thus we have natural morphisms
$f_j:  C_j^i\lra D_j^i$. The $f_j$ is an isomorphism when $j\le i$. From the proof of Proposition \ref{prop-ss}, the $C_i^i$ is not Cartier only when locally $\CX_i$ looks one of the following three situations
$$\begin{cases}R[[x_1, x_2, y_1, y_2, y_3]]/(y_1y_2-\pi, x_1x_2-y_1y_2),\\
R[[x_1, x_2, y_1, y_2, y_3]]/(y_1y_2y_3-\pi, x_1x_2-y_1y_2y_3),\\
R[[x_1, x_2, y_1, y_2, y_3]]/(y_1y_2y_3-\pi, x_1x_2-y_1y_2).
\end{cases}$$
Assume that $C_i^i$ is defined by $(x_1, y_1)$. Then $C_i^i$ has complete local  ring $k[[x_2, y_2, y_3]]$.
Blowing up $(x_1, y_1)$ means gluing two affines by adding rational functions  $x_1/y_1$ and $y_1/x_1$, respectively.
 In the first and third cases, $x_1/y_1=y_2/x_2$. Thus $f_{i+1}: C_{i+1}^i\lra D_{i+1}^i$ is a  blow-up  along the ideal $(x_2, y_2)$.
 In other words, $C_{i+1}^i=D_k^i$ for some $k>i+1$. This implies that $f_t$ is an isomorphism.

In the second case, $x_1/y_1=y_2y_3/x_2$. Thus $f_{i+1}: C_{i+1}^i\lra D_{i+1}^i$ is a blow-up  along the ideal $(x_2, y_2y_3)$. 
For  constructing  $D_t^i$ for some $k>i+1$, we need to   blow-up  ideals $(x_2, y_2)$ and  $(x_2, y_3)$ in some order. 
After these blow-ups, the ideal sheaf  $(x_2, y_2y_3)$ will be invertible.  By universal property of blowing-up, there is a map $g_{i+1}: D^i_t\lra C_{i+1}^i$
such that $f_{i+1}\cdot g_{i+1}$ is the projection $D_t^i\lra D_{i+1}^i$. This implies that $f$ is an isomorphism. 
\end{proof}

 \section{Height pairings}\label{hp}
 
 In this section, we let $K$ be a number field and $X$ be a smooth, geometrically connected,  and projective variety over $K$ of dimension $n$.
 Let $\CX$ be a regular, flat, and projective model over $\CO_K$.
 But note that all results work for the function field $K$ of a smooth and projective curve $B$ over a  field with the following minor modification: 
 \begin{enumerate}
 \item the integral model $\CX$ over $\CO_K$ is replaced by the integral model $\CX$ over $B$, and 
 \item the Chow group $\wh \Ch^*(\CX)$ is replaced by  $\Ch^*(\CX)$.
  \end{enumerate}
 
 Recall that for each integer $i$, we have a cycle class map 
 $$\cl^i: \Ch^i(X)\lra H^{2i}(X)(i).$$ 
Let $A^*(X)$ and $\Ch^i(X)^0$ denote the image and kernel, respectively.
Then Beilinson \cite{Be} and Bloch \cite{Bl}  constructed  a height pairing 
$$\pair{\cdot, \cdot}_\BB: \Ch^i(X)^0\times \Ch^{n+1-i}(X)^0\lra \BR$$
under the conjecture that  {\em Beilinson--Bloch's  condition} always holds: 
{\em $X$ has a regular integral model $\CX$ over $\CO_K$
such for each integer $i$ and each cycle $\alpha \in\Ch^i(X)$ there is an integal cycle $\bar \alpha\in\Ch^i(\CX)$ with the following properties:
\begin{enumerate}
\item  the restriction of $\bar \alpha$ on the generic fiber $X$ is $\alpha$: $ \bar \alpha _K=\alpha$;
\item  for any vertical cycle $\beta\in \Ker ( \Ch^{n+1-i}(\CX)\lra \Ch^{n+1-i}(X))$,
$$\bar \alpha\cdot \beta=0.$$
\end{enumerate}}
Such an integral cycle can be extended into an arithmetic cycle  $\alpha^\flat=(\bar\alpha,  g_\alpha)$ by adding a green current $g_\alpha$ by the equation  $\pp g_\alpha =\delta _{\alpha}$.
The Beilinson--Bloch height pairing is then defined as 
the intersection number 
$$\pair{z, w}_\BB=z^\flat\cdot w^\flat.$$
This pairing does not depend on the choice of integral models and base change and has the following functorial properties:
\begin{enumerate}
\item 
Let $K'$ be a finite extension of $K$, and $X'$ and $X$ are smooth and projective varieties of dimension $n$ over $K'$ and $K$, respectively, such that both $X'$ and $X$ satisfy Beilinson--Bloch condition. Let  $f: X'\lra X$ be a morphism over $K$.
Then for any $\alpha\in\Ch^i(X')^0$ and $\beta\in \Ch^{n+1-i}(X)^0$, we have the following projection formula:
$$\pair {\alpha,  f^*\beta}_\BB=\pair{f_*\alpha, \beta}_\BB.$$
See for example \cite[Lemma 4.1 and Formula (18) on page 787]{Ku96} for more general statements.
\item The Beilinson--Bloch pairing is an extension of the N\'eron--Tate height pairing $\pair{\cdot, \cdot}_\NT$ on $\Pic ^0(X)\times \Alb (X)$ in the following sense:
for any $\alpha \in \Ch^1(X)^0, \beta \in \Ch^n(X)^0$, we have 
$$\pair {\alpha, \beta}_\BB=-\pair{\alpha, \AJ (\beta)}_\NT .$$
\item For any $\alpha\in  \Ch^i(X)^0, \beta \in \Ch^j(X)^0, \gamma \in \Ch^k(X)$ such that $i+j+k=n+1$.
Then
$$\pair{\alpha\cdot \gamma,\beta}_\BB=\pair{\alpha, \gamma \cdot \beta}_\BB.$$
Indeed, for any lifting $\wh \gamma \in\wh\Ch^k(X)$, we have
$$(\alpha\cdot \gamma )^\flat=\alpha ^\flat\cdot\wh \gamma, \qquad (\gamma\cdot \beta)^\flat=\wh\gamma \cdot \beta^\flat.$$
It follows that 
$$\pair{\alpha\cdot \gamma,\beta}_\BB=\wh \gamma \cdot \alpha ^\flat\cdot \beta ^\flat=\pair{\alpha, \gamma \cdot \beta}_\BB.$$
\end{enumerate}

In the rest of this section, we assume that $X=C\times S$ is a product of a curve and surface, both smooth and projective. 
The  key  result of this paper is as follows:
\begin{thm} Assume that $C$ and $S$ have strict semistable model on $\Spec \CO_K $.
Then $X$ satisfies Beilinson- Bloch's condition. In particular, the  Beilinson--Bloch height pairing is well-defined on $X$.
\end{thm}
\begin{proof} By assumption,  $C$ and $S$ have strictly semistable models $\CC$ and $\CS$ over $\CO_K$.  By Proposition \ref{prop-ss}, the blow-up $\CC\times_{\CO_K} \CS$ along its irreducible components in the special fiber
gives a strictly semistable model $\CX$ for $X$. By 
 Corollary \ref{cor-sc},    and Proposition
 \ref{prop-cb}, we know that each irreducible component of the special fiber of $\CX$ satisfies the standard conjecture. 
 Then, we apply Theorem 1.5.1 and Proposition 1.6.3 in \cite{Zh21}  to get that $\CX$ satisfies the Beilinson--Bloch condition.
 \end{proof}
 
 \begin{remark} The idea of using standard conjecture to define height pairing was first used by K\"unnemann \cite{Ku98a} based on 
 early work of Bloch, Gillet, and Soul\'e \cite{BGS}. 
 K\"unnemann's work \cite{Ku98a} covers the case of abelian varieties over local fields with total degeneration and the case of varieties uniformized by the Drinfeld upper-half spaces. Furthermore, in \cite{Ku98b}, K\"unnemann extended his work to 
 general abelian varieties by constructing appropriate regular integral models.
 \end{remark}
 
 Now, we want to extend the definition of the Beilinson--Bloch height pairing to the case where $C$ or $S$ does not have strictly semistable 
 models over $\CO_K$. 
 Using de Jong's alteration \cite{dJ}, there is a finite extension $K'$ of $K$ and generically finite morphisms $f_C: C'\lra C_{K'}$ 
 and $f_{S}: S'\lra S_{K'}$ from smooth and projective curve and surface respectively such that $C'$ and $S'$ 
 both have strictly semistable reduction over $\CO_{K'}$. Let $X'=C'\times S'$ and $f: X'\lra X$ the induced morphism. 
 Then we define a Beilinson--Bloch height pairing on $\Ch^i(X)^0$ by the formula:
 $$\pair {\alpha, \beta}_\BB:=\frac 1{\deg f} \pair {f^*\alpha, f^*\beta}_\BB, \qquad \alpha \in \Ch^i(X)^0, \beta \in \Ch^{4-i}(X)^0.$$
This definition does not depend on the choice of covering $X'\lra X$ because of the following two facts:
\begin{enumerate}
\item A third one will dominate any two such coverings. 
\item If $g:X''\lra X'$ is a generically finite morphism from another variety $C''\times S''$ over a finite extension $K''$ of $K$ with strictly semistable models  over $\CO_{K''}$,
then for any $\alpha \in \Ch^i(X')^0, \beta \in \Ch^{4-i}(X')^0$, 
 $$ \pair {g^*\alpha, g^*\beta}_\BB=\pair {g_*g^*\alpha, \beta}_\BB=\deg g \pair {\alpha, \beta}_\BB.$$
 In the last step, we used the projection formula from intersection theory \cite[Example 8.1.7]{Fu} to conclude that for a generically finite proper map between regular varieties $f_*f^*\alpha =(\deg f) \alpha $.
\end{enumerate}

We want to translate the height pairing to the intermediate Jacobians $J^i(X)$ defined in \S\ref{sec-htc}.
\begin{prop} The  Beilinson--Bloch pairing is induced by a pairing  on 
$$J^i(X)\times J^{4-i}(X)\lra \BR.$$
\end{prop}

\begin{proof}
We need only check the case $i=2$: for all  $\alpha, \beta\in \Ch^1(X)^0,\gamma \in \Ch^2(X)^0$, we have identity
$\pair{\alpha\cdot\beta, \gamma}_\BB=0$. 
This identity follows from the  adjoint formula and relation  with the N\'eron--Tate formula and Lemma \ref{lem-JCh}:
$$\pair{\alpha\cdot\beta, \gamma}_\BB=\pair{\alpha, \beta\gamma}_\BB
=-\pair{\alpha,  \AJ (\beta\gamma)}_\NT=0.$$
\end{proof}

\begin{remark}Let $f: X'\lra X$ and  $\xi', \eta'$ be as in Remark \ref{rem-fc}. Then for any $\alpha\in J^i(X), \beta\in J^{4-i}(X')$,
 we have a projection formula:
 $$\pair{f^*\alpha, \beta}_\BB=\pair{\alpha, f_*\beta}_\BB.$$
 More generally for any $\gamma \in \Ch^3(X\times X)$, we have
  \begin{equation}
  \label{eq-proj}\pair{\gamma ^*\alpha, \beta}_\BB=\pair{\alpha, \gamma_*\beta}_\BB.\end{equation}
  This can be proved using operational arithmetic Chow groups of Gillet-Soule. See the proof of \cite[Formula (18),
  page 787]{Ku96}. More precisely, let $\wh\gamma  \in \wh \Ch_3(\CX\times \CX)$ be a lifting of $\gamma$ as an element in the homological  Chow group. Then, in terms of 
  $$(\gamma ^*\alpha)^\flat=p_{1*} (\alpha ^\flat \cdot _{p_2}\wh \gamma),
  \qquad (\gamma _*\beta )^\flat=p_{2*}(\beta^\flat\cdot _{p_1}\wh \gamma ).$$
  Thus, the equation \ref{eq-proj} becomes 
  $$\wh\deg (\beta^\flat \cdot_{p_1} (\alpha ^\flat \cdot _{p_2}\wh\gamma))
 = \wh\deg (\alpha ^\flat  \cdot_{p_1}(\beta^\flat\cdot _{p_2}\wh\gamma)).$$
 \end{remark}

Let $\CL\in \Ch^1(X)$ be ample cycle.  The intersection with $c_1(\CL)$ then defines an operator 
$$\sL: J^i(X)\lra J^{i+1}(X).$$
Notice that $\sL^2: J^1(X)\lra J^3(X)$ is a polarization. Thus  $\sL^2$ is an  isomorphism, and 
for any $\alpha\in J^1(X)$, 
$$-\pair{\alpha, \sL^2 \alpha}_\BB>0.$$

For $J^2(X)$, let $J^2(X)_0$ denote the kernel of $\sL$ called the primitive part of $J^2(X)$. Then, we have a decomposition:
$$J^2(X)=\sL J^1(X)\oplus J^2(X)_0.$$
Here is an arithmetic Hodge index conjecture:
\begin{conj}[Beilinson]
The Beilinson--Bloch height pairing on $J^2(X)_0$ is positive definite. 
\end{conj}

 \section{Arithmetic  diagonals}\label{sec-ad}
 This section assumes that $K$ is a number field or a function field
 and that $f: C\lra S$ is a morphism from a smooth and projective curve to a smooth and projective surface.
 We assume that $[C: f(C)]=1$. Let $\Gamma$ be the graph of $f$ in $X:=C\times S$.
 Let $\xi$ be an ample class in $\Ch^1(S)$ such that $\eta:=f^*\xi$ has degree $1$.
 Then by Corollary \ref{cor-ChJ}, we obtain a modified class in $J^2(X)_{00}$ given by 
 $$\gamma:=\Gamma -[C]\boxtimes f_*\eta-\eta\boxtimes f_*[C]-\mu_f\cdot \xi$$
 where $\mu_f\in \Ch^1(X)_{00}$ is characterized by the following identities in $\Alb (C)$ and $\Alb (S)$: 
 $$\mu _f^*(\xi^2)=0, \quad f_*(p)-f_*\eta=\mu_{f*}(p)\cdot \xi, \quad\forall p\in C.$$
 We call $\gamma$ an {\em arithmetic diagonal}.  
 
 \begin{example}
Assume that $f: C\lra S=C\times C$ is the diagonal embedding to its self-product. Then $X=C\times C\times C$.
Let  $\xi=\frac 12e\times C+\frac 12 C\times e$  be an ample class with   $e\in \Ch^1(C)$ of degree $1$. Then $\eta=e$.
In this way, $\gamma$ is the modified diagonal defined by Gross and Schoen \cite{GS95}. They showed that $\gamma =0$ if 
 $C$ is rational, elliptic, or hyperelliptic. 
 When $g(C)\ge 2$, the height of this modified diagonal has been computed in our previous paper \cite{Zh10}, Theorem 1.3.1:
$$\pair{\gamma, \gamma}=\frac{2g+1}{2g-2}\omega ^2+\pair{x_e, x_e}_\NT+\sum _v \varphi (X_v)\log \RN (v).$$
where $\omega$ is the admissible relative dualizing sheaf of $C$, $x_e=e-K_c/(2g-2)\in \Ch^1(C)^0$, and 
$v$ runs through the set of places of $K$, and $\varphi(X_v)$ are local invariants of $v$-adic curve $X_v$.
The $\NS(S)_0$ is generated by $e\times C-C\times e$ and the subgroup $\NS(S)_{00}$ of classes of line bundles whose restriction on 
$e\times C$ and $C\times e$ are both trivial. For each $h\in \NS(X)_0$, $f^*h$ is the divisor of fixed points of $h$.
\end{example}

 By Proposition \ref{prop-000}, we can also compute the projection of $\gamma$ to the subspace $\Pic^0(C)\boxtimes \NS (S)_0$ with respect to an orthonormal base
 $h_i$ with coefficient in $\BR$ in the sense $\pair{h_i, h_j}_\NS=-\delta _{ij}$:
 $$\sum_i \gamma^*h_i\boxtimes h_i=\sum _if^*_\eta h_i\boxtimes h_i, \qquad f^*_\eta h_i:=(f^*h_i-(\deg f^*h_i)\eta)\boxtimes h_i.$$
 Thus, Beilinson's Hodge index conjecture implies the following inequality:
$$\pair{\gamma, \gamma}_\BB\ge \sum _i \pair{f_\eta^*h_i, f_\eta^*h_i}_\NT.$$

 \begin{remark}[Correspondences and arithmetic GGP]
 We can define a family of cycles $t_*\gamma$ indexed by elements by rigid  $t\in \Ch^3(X\times X)_\rig$ as defined in Remark \ref{rem-fc}. 
 When $f: C\lra S$ is an embedding  of a Shimura curve into a Shimura surface, and  $\xi$,  $\eta$ are  Hodge classes, our construction of height pairings $\pair{t_*\gamma, t_*\gamma }_\BB$ gives an unconditional formulation of  the arithmetic Gan--Gross--Prasad conjecture \cite{GGP, Zh18, Zh19}
which relates the heights of these cycles to the derivative of the $L$-series. In the triple product case,  $f$ is the diagonal embedding $\Delta: C\lra S=C\times C$, already studied in  Gross--Kudla \cite{GK}, and in our recent work \cite{YZZ}. 
 In the general case, we take Hecke eigenforms $\alpha\in \Gamma (C, \Omega _C^1)_\BC$ and $\beta\in \Gamma (S, \Omega _S^2)_\BC$ for some embedding $K\lra \BC$, 
 and Hecke operators 
 $t$ with complex  coefficients whose action on $\Gamma (X, \Omega _X^3)_\BC$ is an  idempotent with image $\BC \alpha\beta$. Then, by GGP conjecture, 
 $\pair{t_*\gamma, t_*\gamma}_\BB$ will be related to the derivative  of certain Rankin--Selberg $L$-series $L(\pi_\alpha\times \pi_\beta, s)$ at its center of symmetry.
 Here is a simple example: take $B$, an indefinite quaternion algebra over $\BQ$ (could be $M_2(\BQ)$), and $F$, a real quadratic field embedded in $B$. Let $\CO_B$ be a maximal order of $B$ containing $\CO_F$. Then we have the Hilbert moduli surface $S$ parameterizing polarized abelian surfaces $A$ with action by $\CO_F$ with some level structure, and the Shimura curve $C$ inside $S$ parameterizing polarized abelian surfaces 
 with action by $\CO_B$.
 \end{remark}

In the rest of this paper, we assume that $H^1(S)=0$. Then, we can write 
 $$\gamma =\Gamma -[C]\boxtimes f_*\eta-\eta\boxtimes f_*C.$$
 More generally, for any  cycle  $e\in \Ch^1(C)$ of degree $1$, and define
 $$\gamma _e= \Gamma -[C]\boxtimes f_*e-e\boxtimes f_*C\in J^2(X).$$
 This class is killed by any polarization $\eta\in \Ch^1(C)$ but not necessarily by polarization $\xi\in \Ch^1(S)$.
 It is killed by  $\xi$ if and only if $e=f^*\xi$. 
 For two different $e_1, e_2\in \Ch^1(C)$ of degree $1$,  the difference is given by 
 $$\gamma _{e_2}=\gamma_{e_1} +(e_1-e_2)\boxtimes f_*C.$$
 Let $\varphi=f^*f_*C\in \Ch^1(C)$ and $d=\deg \varphi$.
 Then, we can compute the difference between their  heights as  follows:
\begin{align*}
 \pair{\gamma _{e_2}, \gamma _{e_2}}_\BB&=\pair{\gamma_{e_1}, \gamma_{e_1} }_\BB-2\pair{\gamma_{e_1}^*f_*C,  e_1-e_2}_\NT
 -\pair {e_2-e_1, e_2- e_1}_\NT d\\
 &=\pair{\gamma_{e_1}, \gamma_{e_1} }_\BB-2\pair{\varphi-de_1,  e_1-e_2}_\NT-\pair {e_1-e_2, e_1-e_2}_\NT d\\
\end{align*}

If $d=0$, we have 
$$
 \pair{\gamma _{e_2}, \gamma _{e_2}}_\BB=\pair{\gamma_{e_1}, \gamma _{e_1}}_\BB-2\pair{\varphi,  e_1-e_2}_\NT
$$
If $\varphi=0$, then the height of $\gamma _e$ does not depend on the choice of $e$.

If $d\ne 0$, we take $e_0=\varphi/d$, to obtain for any $e$, 
$$
 \pair{\gamma _e, \gamma _e}_\BB=\pair{\gamma_{e_0}, \gamma_{e_0} }_\BB-d\pair{e-e_0,  e-e_0}_\NT
$$
Thus $\pair{\gamma _e, \gamma _e}_\BB$ has its maximal value at $e_0$ if $d>0$ and minimal value at $e_0$ if $d<0$.

\begin{example} Take $S=\BP^2_B$ with a hyperplane $H$, and  and  $e_S=H^2\in \Ch^2(S)$.
 Let $f: C\lra S$ be an embedding from a smooth curve. 
Then, as a Chow cycle
$$\Delta _S=e_S\boxtimes S+S\boxtimes e_S+H\boxtimes H, \qquad \Gamma=C\boxtimes e_S+f ^*H\boxtimes H$$
It follows that 
$ \gamma =0$ for any polarization $\xi$ on $S$.
\end{example}

\begin{example}
Take $S=\BP^1\times \BP^1$ with two hyperplanes $H_1=0\times \BP^1, H_2=\BP^1\times \infty$. Then $\NS (S)=\BQ H_1+\BQ H_2$.
Let $f: C\lra S$ be an embedding from a smooth curve linearly equivalent to   $aH_1+bH_2$ with both $a, b>0$.
Then we can define an orthogonal basis for $\NS(S)=\Ch^1(S)$:
$$h_1=[\iota C]=aH_1+bH_2, \qquad h_2=aH_1-bH_2.$$
Thus, as Chow cycles, we have the following identity:
$$\Delta _S=e_S\boxtimes S+S\boxtimes e_S+h_1\boxtimes h_1/(h_1, h_1)+h_2\boxtimes h_2/(h_2, h_2)$$
where $e_S=0\times \infty \in \Ch^2(S)$.
Take $\xi=[C ]/[C]^2$ and $\eta=f^*\xi$. Then 
 $$\Gamma =C\boxtimes e_S+\eta\boxtimes f (C)+\iota ^*h_2\boxtimes h_2/(h_2, h_2).$$
 It follows that 
 $$\gamma =f ^*h_2\boxtimes h_2/(h_2, h_2).$$
It is clear that $\alpha:=f ^*(h_2)$ has degree $0$. In fact, let $\pi_1, \pi_2: C\lra \BP^1$ be two projections induced by $f$,
then $\alpha=a\pi_1^*\infty-b\pi_2^*\infty$. Thus 
$$\pair{\gamma, \gamma }_\BB=-\pair {\alpha, \alpha}_\NT/(h_2, h_2)=\frac 1{2ab} \pair {\alpha, \alpha}_\NT.$$
Moreover, the projection of $\gamma$ on $J^2(X)_{00}=\Alb (C)\boxtimes \NS (S)_0$ is given by 
$$\alpha\boxtimes h_2/(h_2, h_2).$$
\end{example}

\section{Unramified calculation}\label{uc}
Now we assume that $K=k(B)$ is the function field of a smooth and projective  curve $B$ over an algebraically closed field $k$, 
and   $f: C\emb S$ an embedding from 
a smooth and proper curve to a smooth and proper surface over $K$. 
We assume that $d:=\deg f^*f_*C\ne 0$. As in \S\ref{sec-ad}, let 
 $e=f^*f_*C/d\in \Ch^1(C)$ of degree $1$, and define a homologically trivial cycle
 $$\gamma:=\Gamma (f)-C\boxtimes f_*(e)-e\boxtimes f_*C\in J^2(X).$$
We want to compute its heights when $f$ extends to a morphism
$\bar f: \CC\lra \CS$ of a smooth and proper curve over $B$ to a smooth and proper surface over $B$.
Denote 
$$\omega _\CC:=\Omega_{\CC/B}^1, \quad \omega _\CS=\wedge ^2\Omega _{\CS/B}^1.$$

\begin{thm}\label{thm-gamma-bb}
$$\pair{\gamma, \gamma}_\BB=
-\frac 1d ((d+1) \omega_{\CC}-\bar f^*\omega _S)\cdot (\omega_\CC-\bar f^*\omega _\CS).$$
\end{thm}

\begin{proof}
 First we want to extend $\gamma$ to a cycle on the  four-fold $\CX=\CC \times _B\CS$
 $$\bar \gamma =\Gamma(\bar f) -\bar e\boxtimes \bar f_*(\CC)-\CC\boxtimes \bar f_* {\bar e}\in \Ch^2(\CX).$$
 where 
$$\bar e:=\bar f^*\bar f_*\CC /d +\kappa F\in \Ch^1(\CC).$$
where $F$ is the fiber class of $\CC$ over $B$ and  $\kappa =-\frac 12 (f^*f_*\CC/d)^2$ so that $\bar e^2=0$ on $\CC$.

Let $i: \CY:=\CC\times_B \CC\lra \CX$ denote the embedding induced by $f$ and let $\delta$ denote the cycle $\Delta-e\times \CC-\CC\times e$.
Then we have $\bar \gamma =i_*\delta$. Thus 
\begin{align*}
&\pair{\gamma, \gamma}_\BB=i_*\delta \cdot i_*\delta =\delta\cdot   i^*i_*\delta
=\delta  \cdot \delta \cdot  i^*i_*\CY =\delta \cdot \delta \cdot i^*\pi_S^* \bar f_* \CC\\
=&\delta \cdot \delta \cdot p_2^*\bar f ^* \bar f_*\CC =p_{2*}(\delta\cdot \delta) \cdot \bar f^*\bar f_*\CC
\end{align*}

We need to compute $p_{2*}(\delta\cdot \delta)$:
$$
p_{2*}(\delta\cdot \delta)=p_{2*}(\Delta \cdot (\Delta-2\bar e\boxtimes \CC-2\CC\boxtimes \bar e))+p_{2*}((\bar e\boxtimes C+C\boxtimes \bar e)^2).
$$

For the first term, we use the projection formula for the diagonal embedding:
\begin{align*}
p_{2*}(\Delta \cdot (\Delta-2\bar e\boxtimes \CC-2\CC\boxtimes \bar e))&=p_{2*}\Delta _*\Delta ^*(\Delta-2e\boxtimes \CC-2\CC\boxtimes \bar e))\\
&=\Delta ^*(\Delta-2\bar e\boxtimes \CC-2\CC\boxtimes \bar e))\\
&=-\omega_\CC -4\bar e.
\end{align*}

For the second term, we compute it directly,
$$
p_{2*}((\bar e\boxtimes \CC+\CC\boxtimes \bar e)^2)=p_{2*}(\bar e^2\boxtimes  \CC+2\bar e\boxtimes  \bar e+\CC\boxtimes \bar e^2)=2\bar e.
$$
Thus we have
$$p_{2*}(\delta\cdot \delta)=-\omega_\CC -2\bar e.$$

Put these two terms together, we get 
$$\pair{\gamma, \gamma}_\BB=-(\omega_\CC+2\bar e)\cdot \bar f^*\bar f_*\CC=-\omega_{\CC}\cdot \bar f^*\bar f_*\CC+2d\kappa 
=-\omega_{\CC}\cdot \bar f^*\bar f_*\CC-\frac 1d (\bar f^*\bar f_*\CC)^2.$$

The final formula in Theorem \ref{thm-gamma-bb} follows from the adjunction formula:
$$f^*f_*\CC=\omega _\CC-f^*\omega_S.$$
\end{proof}

\begin{example}
We assume that $\CS/B$ is a smooth family of $K3$ surfaces. Then $\omega _\CS$ is a vertical class $h(S) F$, for 
the canonical height  
$h(S)\in \BR$ of $\CS$.
Then $f^*f_*^*\CC=\omega _\CC-h(S) F$ with degree $d=2g-2$.
It follows that 
$$\pair{\gamma, \gamma}_\BB=-\frac {2g-1}{2g-2}\omega_\CC^2+2g h(S) .$$
If $g\ge 2$, then by Beilinson's index conjecture,  $\pair{\gamma, \gamma}_\BB\ge 0$. Thus, we have the following conjectured inequality: 
$$\omega _\CC^2\le \frac {4g(g-1)}{2g-1}h(S).$$
\end{example}


\begin{thebibliography}{[AB]}


\bibitem{Be}
A. Beilinson,  {\em Height pairing between algebraic cycles.}
In:  Current Trends in Arithmetical Algebraic
Geometry, Arcata, Calif., 1985. Contemp. Math., vol. 67, pp. 1--24. Amer. Math. Soc., Providence (1987)

\bibitem{SGA}
P.  Berthelot,  A. Grothendieck,  and L. Illusie (eds), 
 S\'eminaire de G\'eom\'etrie Alg\'ebrique du Bois Marie
(1966-67) {\em Th\'eorie des intersections et th\'eor\'eme de Riemann-Roch,} (Lecture notes in mathematics
225). Springer, Berlin

\bibitem{Bl}
 S. Bloch,  {\em Height pairings for algebraic cycles.}  J. Pure Appl. Algebra 34(2-3), 119-145 (1984). Proceedings of the Luminy Conference on Algebraic K-theory, Luminy, 1983

\bibitem{BGS}
S. Bloch, H. Gillet,  and C. Soul\'e,  {\em Algebraic cycles on degenerate fibers.}
  Arithmetic geometry (Cortona, 1994), 45--69, Sympos. Math., XXXVII, Cambridge Univ. Press, Cambridge, 1997. 
  
   \bibitem{dJ}
 J. de Jong,  {\em Smoothness, semi-stability and alterations.}
  Inst. Hautes \'Etudes Sci. Publ. Math. No. 83 (1996), 51-93.

\bibitem{Fu}
W. Fulton, 
{\em Intersection theory.}
Ergeb. Math. Grenzgeb. (3), 2[Results in Mathematics and Related Areas (3)]
Springer-Verlag, Berlin, 1984. xi+470 pp.

\bibitem{GGP}
W. T. Gan, B.  Gross, and D. Prasad, 
{\em Symplectic local root numbers, central critical L-values, and restriction problems in the representation theory of classical groups,}
 In: Sur les conjectures de Gross et Prasad, Ast\'erisque 346 (2012).
 
 \bibitem{GS91}
H. Gillet and C. Soul\'e,  {\em Arithmetic intersection theory.}  Inst. Hautes \'Etudes Sci. Publ. Math. No. 72 (1990), 93-174 (1991).


\bibitem{GS95}
B. Gross and C.  Schoen, 
{\em The modified diagonal cycle on the triple product of a pointed curve. }
Ann. Inst. Fourier (Grenoble) 45 (1995), no. 3, 649--679.

\bibitem{GK}
B.  Gross and S. Kudla, {\em Heights and the central critical values of triple
product L-functions.}
Compositio Math. 81 (1992), no. 2, 143--209.

\bibitem{Ha}
U. T. Hartl, 
{\em Semi-stability and base change.} Arch. Math. (Basel) 77 (2001), no. 3, 215--221.

\bibitem{Ku96}
K. K\"unnemann, 
{\em Higher Picard varieties and the height pairing,}
Amer. J. Math. 118 (1996), no. 4, 781–797.

\bibitem{Ku98a}
 K. K\"unnemann, 
 {\em The K\"ahler identity for bigraded Hodge-Lefschetz modules and its application in non-Archimedean Arakelov geometry. }
 J. Algebraic Geom. 7 (1998), no. 4, 651-672. 
 
  \bibitem{Ku98b}
K. K\"unnemann,  {\em Projective regular models for abelian varieties, semistable reduction, and the height pairing. }
 Duke Math. J. 95 (1998), no. 1, 161-212.

 \bibitem{YZZ}
X. Yuan,  S.  Zhang, and W. Zhang, 
{\em Triple product L-series and
Gross--Kudla--Schoen   cycles},
in preparation.

\bibitem{Zh10}
S. Zhang, 
{\em Gross--Schoen cycles and Dualising sheaves.} Invent. Math. 179 (2010),
no. 1, pp 1--73.

\bibitem{Zh19}
S. Zhang, 
{\em Linear forms, algebraic cycles, and derivatives of L-series.}
 Sci. China Math. 62, 2401--2408 (2019). 


\bibitem{Zh21}
S. Zhang,  {\em Standard Conjectures and Height Pairings}. 
 Pure and Applied Math. Quart., 18 (2022),
no. 5, 2221-- 2278.


\bibitem{Zh18}
W. Zhang, 
{\em Periods, cycles, and $L$-functions: a relative trace formula approach.}  Proceedings ICM-2018






\end{thebibliography}
\end{document}